\documentclass[10pt]{amsart}
\usepackage{amsmath,amssymb}
\usepackage{graphicx,subfigure}

\usepackage{color}

\newcommand{\diver}{\mathop{\rm div}\nolimits}

\newcommand{\p}{\partial}

\newcommand{\vk}{\varkappa}

\newcommand{\vp}{\varphi}

\newcommand{\cI}{\mathcal I}
\newcommand{\STR}{\mathbb S^3}
\newcommand{\z}{\frak z}

\newcommand{\R}{{\mathbb R}}
\newcommand{\C}{{\mathbb C}}

\newcommand{\Z}{{\mathbb Z}}

\newcommand{\T}{{\mathbb T}}

\newcommand{\es}{{\mathbb S}}

\newcommand{\cP}{{\mathcal P}}

\newcommand{\const}{\mathop{\rm const}\nolimits}

\newcommand{\rot}{\mathop{\rm rot}\nolimits}

\newcommand{\SVect}{\text{\upshape SVect}}

\def\12{\tfrac12}

\def\lan{\langle}
\def\ran{\rangle}

\def\vp{\varphi}

\newcommand{\om}{\omega}
\newcommand{\intt}{\frac1{2\pi}\int_0^{2\pi}}
\newcommand{\al}{\alpha}

\newtheorem{theorem}{Theorem}[section]
\newtheorem{proposition}[theorem]{Proposition}
\newtheorem{corollary}[theorem]{Corollary}

\newtheorem{lemma}[theorem]{Lemma}

\theoremstyle{definition}
 \newtheorem{definition}[theorem]{Definition}

\newtheorem{example}[theorem]{Example}

\newtheorem*{definition*}{Definition}

\theoremstyle{remark}
\newtheorem{remark}[theorem]{Remark}

\numberwithin{equation}{section}

\newcounter{bk}

\newcounter{sk}

\newcounter{dps}

\usepackage{lipsum}

\begin{document}

\title[Global, local and dense non-mixing of the 3D Euler equation]{Global, local and dense non-mixing \\ of  the 3D Euler equation }

\author{Boris Khesin}
\address{Department of Mathematics, University of Toronto, Toronto, ON M5S 2E4, Canada}
\email{khesin@math.toronto.edu}

\author{Sergei Kuksin}
\address{Universit\'e Paris-Diderot (Paris 7), UFR de Math\'ematiques - Batiment Sophie Germain, 5 rue Thomas Mann, 75205 Paris CEDEX 13, France  \& School of Mathematics, Shandong University, Jinan, PRC \& Saint Petersburg State University, Universitetskaya nab., St. Petersburg, Russia}
\email{sergei.kuksin@imj-prg.fr}

\author{Daniel Peralta-Salas}
\address{Instituto de Ciencias Matem\'aticas, Consejo Superior de
 Investigaciones Cient\'\i ficas, 28049 Madrid, Spain}
\email{dperalta@icmat.es}

\maketitle

\begin{abstract}
 We prove  a non-mixing property of the flow of the 3D Euler equation which has a local nature:
 in any neighbourhood of a ``typical'' steady solution there is a generic set of initial conditions, such that the corresponding Euler flows will never enter a vicinity of the original steady one.
 More precisely, we establish that
 there exist stationary solutions $u_0$ of the Euler equation on $\mathbb S^3$ and divergence-free vector fields $v_0$ arbitrarily close to $u_0$, whose (non-steady) evolution by the Euler flow cannot converge in the $C^k$ H\"older norm ($k>10$ non-integer) to any stationary state in a small (but fixed a priori) $C^k$-neighbourhood of $u_0$. The set of such initial conditions $v_0$ is open and dense in the vicinity of $u_0$.
 A similar (but weaker) statement also holds for the Euler flow on $\mathbb T^3$.
 Two essential ingredients in the proof of this result are a geometric description of all steady states near certain nondegenerate stationary solutions, and a KAM-type argument to generate knotted invariant tori from elliptic orbits.
\end{abstract}

\section{Introduction}\label{s0}

The dynamics of an ideal incompressible fluid on a Riemannian manifold is described by the Euler equation. It is an infinite-dimensional Hamiltonian system and has many peculiar properties, impossible in their finite-dimensional counterparts.
For instance, in 2D the Euler equation possesses wandering solutions, which never return to the vicinity of the initial condition~\cite{Na91, Sh97}, while the Poincar\'e recurrence theorem would guarantee the return in finite-dimensional systems with  convex Hamiltonians.
In 3D the Euler equation has a global non-mixing property: there are two open sets of fluid velocity fields such that the solutions with initial conditions from one of these sets will never enter the other set~\cite{KKP14}. In the present paper we prove that the non-mixing property has a local (and dense) nature: such two sets can be found in any neighbourhood of a ``typical" steady solution, as we explain below. Thus this does not only establish ubiquitous appearance of non-mixing in the phase space, but this can also be thought of as a step towards the wandering property of Euler solutions in 3D, with the existence of
solutions non-returning to a nearby neighbourhood of the initial conditions, rather than its own.

Recall that the dynamics of an ideal fluid flow on a Riemannian 3-manifold $M$ is described by the Euler equation
on the fluid velocity field $u(\cdot,t)$:
\begin{equation} \label{eq:euler}
\partial_tu+ \nabla_u u=-\nabla p\,, \;\; \diver u=0\,.
\end{equation}
Here $\nabla_u u$ is the covariant derivative of the vector field $u$ along itself,  $\diver$ is the divergence operator computed with the Riemannian volume form, and $p(\cdot,t)$ is the pressure function, uniquely defined by the equations up to a constant.

For a closed manifold $M$ (i.e., compact and without boundary), the Euler equation defines a local flow $\{\mathcal S_t\}$
of  homeomorphisms of the H\"older space of $C^k$ divergence-free vector fields on $M$ provided that $k>1$ is not an integer~\cite{EM70}. Accordingly, for any $u_0$ in this space, the solution
$$
u(t,\cdot)=\\\mathcal S_t(u_0)
$$
with the initial condition $\mathcal S_0(u_0)=u_0$
is defined for $-t_*(u_0)<t<t^*(u_0)$ and is $C^1$-smooth in $t$. This property fails for integer $k$~\cite{Bo}.

The Kelvin circulation law is the following remarkable property of the Euler equation:
the \emph{vorticity field} $\omega:=\rot u$
is transported by the fluid flow (see Appendix~\ref{app1}), i.e.
\begin{equation}\label{eqtrans}
\partial_t\omega=[\omega,u]:=\nabla_{\omega}u-\nabla_u\omega\,,
\end{equation}
and hence the vortex lines at $t=0$ are diffeomorphic to the corresponding
vortex lines at any other $t$ for which the solution exists.

Inspired by the phenomenon of the vorticity transport, we introduced in~\cite{KKP14} an integral of motion
for the 3D Euler equation that is independent of the energy and  helicity. This conserved quantity is a functional $\vk$ on the space of divergence-free vector fields which measures the fraction of $M$  covered by ergodic invariant tori of $\rot u$. Moreover,
the way the invariant tori are embedded (knotted) in $M$ is also an invariant
and it gives a family  $\vk_a$, $a\in\Z$, of infinitely many conserved quantities.
(More precisely, the index $a$ belongs to the countable set of different embedding classes.)
The whole family of quantities $\{\vk_a\}$ for a given divergence-free vector field is called the integrability spectrum of this field on the manifold.

These integrability functionals are invariant under arbitrary volume-pre\-serving diffeomorphisms. On the other hand, they are not even continuous, so this does not contradict the fact that the helicity is essentially the only $C^1$ Casimir of the corresponding coadjoint action~\cite{EPT16}.
Nevertheless, thanks to the KAM theory, the functional $\vk_a$ has good continuity properties when computed for nearly-integrable nondegenerate vector fields, a property that was exploited in~\cite{KKP14} to analyse the evolution of the Euler equation.

To formulate the main theorem  of this paper we recall that a stationary (or steady) flow is a divergence-free
solution $u$ of the stationary Euler equation $\nabla_u u=-\nabla p$.
 In 3D there is a particularly useful equivalent way to write the equation for steady states:
  a divergence-free field $u$ is a steady Euler flow if there is a (Bernoulli) function $\mathcal B$ on $M$ such that $u \times \rot  u = \nabla  {\mathcal B}$, where the operations $\times$ and $\rot$ are taken with respect to the Riemannian metric on $M$; see  Appendix~\ref{app1}.  In particular, this implies that the vector fields $u$ and $\rot u$ commute on $M$. In the statement of the theorem we also use the notation
$$
B_\epsilon(w;k):=\{h\in \SVect^k(M):\|h-w\|_{k}<\epsilon\}
$$
for the $\epsilon$-neighbourhood in the $C^k$ H\"older norm of a divergence-free vector field $w$. Here, $\SVect^k(M)$ denotes the space of divergence-free $C^k$-smooth vector fields on $M$, endowed with the $C^k$ norm $\|\cdot\|_k$.
The following  main results of the paper demonstrate the local nature of the non-mixing property for the Euler equation.
(The class of nondegenerate shear  flows, evoked there, is defined below in Section~\ref{S.3}, see Definitions~\ref{D:TI} and~\ref{D:torus}.)


\begin{theorem}\label{T:main0}
Let $u_0$ be a stationary nondegenerate shear Euler flow on $\mathbb S^3$ or $\mathbb T^3$, and fix a non-integer $k>5$.
Then:
\begin{enumerate}
\item For any $\epsilon>0$ there exists a vector field $v_0\in B_\epsilon(u_0;k)$ and a positive constant $\delta=\delta(v_0)$ such that
$$
\|\mathcal S_t(v_0)-  u_0\|_k>\delta
$$
for  all $t$ for which the Euler flow $\mathcal S_t(v_0)$ is defined.
\item The set of such initial conditions $v_0$ is open and dense in $B_\epsilon(u_0;k)$, provided that $\epsilon$ is sufficiently small.
\end{enumerate}
\end{theorem}

In other words, arbitrarily close to a stationary nondegenerate shear Euler flow $u_0$
there is an initial condition $v_0$, such that the solution $S_t(v_0)$ remains at a positive distance from $u_0$.
 In the case of the sphere $\mathbb S^3$ we  prove a   strengthening of Theorem~\ref{T:main0}.

\begin{theorem}\label{T:main}
For a  stationary nondegenerate shear Euler flow
$u_0$ on $\mathbb S^3$ we fix a non-integer $k>10$ and a sufficiently small constant $\eta>0$ (depending on $u_0$).
Then for any $\epsilon>0$ there exists a vector field $v_0\in B_\epsilon(u_0;k)$ and a positive constant $\delta$ such that
$$
\|\mathcal S_t(v_0)-\widetilde u_0\|_k>\delta
$$
for any stationary Euler flow $\widetilde u_0\in B_\eta(u_0,k)$ and all $t$ for which the Euler flow $\mathcal S_t(v_0)$ is defined. As above, the set of such initial conditions $v_0$ is open and dense in $B_\epsilon(u_0;k)$, provided that $\epsilon$ is sufficiently small.
\end{theorem}

Here the solution $S_t(v_0)$  with the initial condition $v_0$ remains at a positive distance not only from $u_0$, but from
any stationary solution in the a priori prescribed $\eta$-vicinity of $u_0$.
An explicit construction of nondegenerate shear  flows on the sphere 
  implies that   $u_0$  is not isolated  in the sense that there are plenty of stationary solutions
arbitrarily close to it. In fact, we provide (see Theorem~\ref{L:main}) a complete description of all steady states in a $C^k$-neighborhood ($k\geq 7$) of any stationary nondegenerate shear Euler flow on $\mathbb S^3$; this can be understood as a local
  3D version of Choffrut-Sverak's theorem~\cite{Sverak} that characterizes all steady states in a $C^k$-vicinity ($k\geq 12$) of certain nondegenerate stationary solutions of the Euler equation on the $2$-dimensional annulus, but the techniques we use are totally different.

The proof of the theorems above is divided in two steps. First, we show that in a sufficiently small neighborhood of a nondegenerate shear Euler flow there is
an open and dense set of vector fields (initial conditions for the Euler equation) exhibiting a positive measure set of ergodic invariant tori that are nontrivially knotted. They arise as a perturbation of a closed trajectory $\gamma$ of $\rot u_0$ which is a $(p,q)$-torus knot with rational rotation number. We show that, generically, this perturbation leads to the existence of new invariant tori whose cores trace the torus knot $\gamma$ and that are not isotopic to the invariant tori of $\rot u_0$.

Second, using the KAM theorem proved in~\cite{KKP14} and Helmholtz's transport of vorticity, we conclude that the evolution of this initial conditions cannot converge, in the suitable topology, to the shear steady state. In the case of the 3-sphere, the statement is stronger because we are able to characterize all the stationary solutions of the Euler equation in a neighborhood of a nondegenerate shear state. Key to prove these results are a KAM stability theorem for elliptic
points, a novel Vey-type result in finite regularity and a new suspension theorem for divergence-free vector fields.

The restriction on the regularity $k>10$ is due to the fact that in the proof we use a finite regularity Vey-type theorem (see Appendix~\ref{app4}), and the KAM theory for elliptic points (see Appendix~\ref{app3}). This regularity is optimal in the sense that it cannot be weakened following our approach. This ``local non-mixing" behaviour
could be compared with the asymptotic stability (in the $L^2$-sense) recently proved~\cite{BM} for shear flows close to the planar Couette flow of the two-dimensional Euler equation on the cylinder $\mathbb S^1\times\R$,
as well as with the results on wandering solutions of the 2D
Euler equation~\cite{Na91, Sh97} discussed above.

{\it Notation.} For a metric space $X$ and $w\in X$ by $B_\epsilon(w,X)$  we denote the open $\epsilon$-neighbourhood of $w$
in $X$, and abbreviate $B_\epsilon(w,SVect^k(M))=: B_\epsilon(w,k)$. For a space $X$ as above ``generic $x\in W$" means ``any
$x$, belonging to some (fixed) open and dense subset $O\subset X$".

\section{KAM, knots, and the integrability functional}

In this section we review some results from~\cite{KKP14} required to prove the main theorem:
the notion of  nondegeneracy in the context of divergence-free vector fields and the invariance of the integrability functional
for each isotopy class. We denote by $\mu$ the Riemanian volume form on a Riemann
 manifold $M$ (in this paper $\mathbb S^3$ or $\mathbb T^3$), and by $\SVect^k_{ex}(M)$  the space of exact divergence-free vector fields on $M$ of class $C^k$ (by  {\it exact} we mean that the field is the $\rot$ of another divergence-free field on $M$).

\subsection{Integrable divergence-free flows}

Let $w\in \SVect^{k}(M)$ be a vector field ($k\ge 1$) on a closed $3$-manifold $M$. Assume that there is a domain $\Omega\subset M$, invariant under the flow of $w$.

\begin{definition}\label{D:IN}
We say that $w\in \SVect^{k}(M)$
 is \emph{integrable} in the domain $\Omega$ if there are finitely many domains $\Omega_j$ of the form $\T^2\times (a_j,b_j)$ covering $\Omega$, with $0\leq a_j<b_j\leq 1$, and coordinates $(\theta_1,\theta_2,\rho)$ in each $\Omega_j$, where $(\theta_1,\theta_2)\in \T^2=(\R/ 2\pi\Z)^2$ and $\rho\in(a_j,b_j)$, such that $w$ is transverse to the sections $\{\theta_2=\text{const}\}$ at each point of $\overline{\Omega_j}$, and its Poincar\'e map $\Pi$ at the section $\{\theta_2=0\}$ takes the form
\begin{equation}\label{1}
\Pi(\theta_1,\rho)=(\theta_1+W_j(\rho),\rho)
\end{equation}
for some function $W_j(\rho)$ of class $C^{k}$ in $\Omega_j$. A vector field $w$ is called integrable nondegenerate if in addition the function $W_j$ satisfies the twist condition
\begin{equation}\label{2}
|W'_j(\rho)|\geq \tau
\end{equation}
for all $\rho\in(a_j,b_j)$ and some constant $\tau>0$, for all $j$. Obviously, the transversality condition implies that $w$ does not vanish in the closure
$\overline \Omega$.
\end{definition}

\begin{remark}
This definition of an integrable nondegenerate vector field guarantees applicability of the KAM theory
for divergence-free vector fields; see \cite{C-S} and \cite[Section~3]{KKP14}. Definition~\ref{D:IN} is a bit weaker than that we presented in~\cite{KKP14}, but it is better suited for our purposes, and sufficient to apply the KAM theorem stated in~\cite[Theorem 3.2]{KKP14}.
\end{remark}

\subsection{Isotopy classes and the integrability functional}

There  are many non-equivalent ways of embedding a torus in $M$. We say that two embedded tori $T^2_0$ and $T^2_1$ are \emph{isotopic} if there exists a continuous in $t$
 family of embedded tori $T^2_t$, $t\in[0,1]$, connecting $T^2_0$ and $T^2_1$. It is well known that this property is equivalent to the existence of an isotopy $\Theta_t:M\times[0,1]\to M$ such that $\Theta_0=id$ and $\Theta_1(T^2_0)=T^2_1$, see \cite{Hae61}. This equivalence relation defines the set of isotopy classes of embedded tori in $M$, a set that we denote by $\mathcal I(M)$. It is standard that the set of isotopy classes $\mathcal I(M)$ is countable, i.e. $\mathcal I(M)\cong \mathbb Z$.

Define the following \emph{collection of functionals} $\vk_a$, $a\in\mathcal I(M)$, on the space of exact divergence-free vector fields:

\begin{definition}\label{defka} (\cite{KKP14})
Given an isotopy class $a\in \cI(M)$, the \emph{integrability functional}
$$
\vk_a:\SVect^1_{ex}(M)\to [0,1]
$$
assigns to an exact  $C^1$-smooth divergence-free vector field $w$ the inner
measure of the set of ergodic
$w$-invariant two-dimensional $C^1$-tori lying in the isotopy class $a$. The sequence
 $$
  \cI(M) \ni a \mapsto \vk_a(w)
 $$
 is called the {\it integrability spectrum of $w$}.
\end{definition}

This functional on the field $\rot u$ is invariant under the evolution of $u$ according to the Euler equation, as  Kelvin's circulation law suggests. More precisely, since the vorticity is transported by the fluid flow, which is volume-preserving, we have the following~\cite[Theorem 6.1]{KKP14}:

\begin{theorem}\label{t3}
 If $u_0 \in\SVect^k(M)$ for $k>2$ and non-integer, then the functional $\vk_a (\rot \mathcal S_t(u_0))\,$ is constant for all $a\in \cI(M)$. In other words, the integrability spectrum of $\rot u_0$, that is the mapping
$$a\in\mathcal I(M) \mapsto \vk_a(\rot(u_0))\,,$$
is an integral of motion of the Euler equations on the space $\SVect^k(M)$.
\end{theorem}

\section{Shear stationary solutions on the 3-sphere and 3-torus}\label{S.3}

\subsection{Nondegenerate shear flows on the sphere.}
In this subsection we introduce a family of integrable and nondegenerate
 stationary solutions of the Euler equation in $\mathbb S^3$ which will be used to define the stationary states that appear in the statements of Theorems~\ref{T:main0} and~\ref{T:main}.

Represent $\mathbb S^3$ as the unit sphere in $\mathbb R^4$:
$$\mathbb S^3:=\{(x,y,z,\xi)\in\R^4:x^2+y^2+z^2+\xi^2=1\}\,.$$
Following~\cite[Example 4.5]{KKP14}, let us consider the Hopf fields $u_1$ and  $u_2$ on $\es^3$
given by
\begin{equation*}
u_1=(-y,x,\xi,-z)|_{\es^3}\,, \qquad u_2=(-y,x,-\xi,z)|_{\es^3}\,
\end{equation*}
(note that $u_1$ and $u_2$  are tangent to $\STR$).
These fields satisfy the equations $\rot u_1=-2u_1$ and $\rot u_2=2u_2$, and hence they are divergence-free, see Appendix~\ref{App2}.
It is evident that the function
$$\rho:=(x^2+y^2)|_{\es^3}, \quad 0\le \rho\le1\,,
$$
is a first integral of both $u_1$ and $u_2$. Notice that the level set $\{\rho=c\}$ is diffeomorphic to the standard (unknotted)
 torus $\mathbb T^2=\R^2/(2\pi\Z)^2$ in $\mathbb S^3$ provided that $c\in (0,1)$, and the critical set
$\{\rho=0\}\cup \{\rho=1\}=:L_0\cup L_1$, where $L_0$ and $L_1$ are loops in $\STR$,
 is diffeomorphic to the Hopf link~\cite{Adams} in $\mathbb S^3$.

 In particular, $\Omega:= \STR \setminus (L_0\cup L_1)$ is diffeomorphic to $\T^2\times (0,1)$.
 The flows of the fields $u_1$ and $u_2$ together with the function $\rho$ provide the following
 explicit coordinates on this product.  Introduce the complex coordinates in the space
 $\R^4 \simeq \C^2 =\{ (\z_1, \z_2): \z_1 = x+iy, \z_2 = z+i\xi\}$.
 Points  $(\z_1, \z_2) \in \mathbb S^3\subset \C^2$ are parametrised via
 $\z_j =\sqrt \rho_j e^{i\theta_j}$ with $\rho_1+\rho_2 =1$. Hence
 \begin{align}\label{the_coord}
\STR\setminus (L_0 \cup L_1) &\simeq \T^2\times(0,1)\\\notag
&= \{ (\sqrt\rho e^{i\theta_1}, \sqrt{1-\rho} e^{i\theta_2}), \ (\theta_1, \theta_2)\in \T^2, 0<\rho<1\}\,.
\end{align}
In the coordinates $ (\theta_1, \theta_2, \rho)$ we have
\begin{equation}\label{relations}
u_1 = \frac{\p}{\p \theta_1} -\frac{\p}{\p \theta_2},\quad u_2 = \frac{\p}{\p \theta_1} +\frac{\p}{\p \theta_2},
\end{equation}
 and the volume form is
 $\mu=d{\rm Vol}\,{}_{\STR} =\tfrac12 d\theta_1 d\theta_2 d\rho$.

Consider the vector field
\begin{equation}\label{E:steady}
u:=f_1(\rho)u_1+f_2(\rho)u_2\,,
\end{equation}
where $f_1$ and $f_2$ are arbitrary $C^k$ real-valued functions.
Obviously $u$ is divergence-free and the function $\rho$ is a first integral.
The vorticity of this field $u$  turns out to be (see Appendix~\ref{App2}):
\begin{align*}
\rot u&=-[f_1'(\rho)(2\rho-1)+2f_1(\rho)+f_2'(\rho)]u_1\\&+[f_2'(\rho)(2\rho-1)+2f_2(\rho)+f_1'(\rho)]u_2=:A_1(\rho)u_1+A_2(\rho)u_2\,.
\end{align*}
Recall that on a Riemannian manifold $M$ the vorticity field $\omega=\rot u$ is defined by the condition
$i_\omega \mu=d\alpha$, where $\alpha=u^\flat$ is the dual 1-form of the vector field $u$ using the Riemannian metric on $M$: $\alpha(\cdot)=g(u,\cdot)$.

Now one can see that $u$ is a stationary Euler flow on the sphere $\STR$ (for any choice of the functions $f_1, f_2$)
as it commutes with $\rot u$, cf.~\cite{KKP14}.
Indeed, the function $\rho$ is a first integral of both $u$ and $\rot u$, while on each torus $\rho={const}$
these two fields are constant (in the coordinates $(\theta_1, \theta_2)$)
 and hence commute. Note also that the Hopf link $L_0\cup L_1$ and its complement $\Omega$ are invariant under the flows of $u$ and $\rot u$.

\begin{proposition}
There is a choice of functions $f_1$ and $f_2$ such that the vector field $\rot u$ is integrable and nondegenerate on
$\Omega:=\mathbb S^3\backslash (L_0\cup L_1)$ in the sense of Definition~\ref{D:IN}.
\end{proposition}

\begin{proof}
In view of~\eqref{relations} the field $\rot u$ reads in the coordinates $(\theta_1,\theta_2,\rho)$ as:
\begin{equation}\label{E:rot}
\rot u=f(\rho)\partial_{\theta_1}+g(\rho)\partial_{\theta_2}\,,
\end{equation}
where we have defined $f:=A_1+A_2$ and $g:=-A_1+A_2$.

If $\rot u$ does not vanish on $\mathbb S^3$, and we can cover $\Omega$ with finitely many domains such that either $f$ or $g$ do not vanish on the closure of these domains, we can define the Poincar\'e maps using the sections $\{\theta_1=0\}$ or $\{\theta_2=0\}$. These maps have the form $\Pi(\theta,\rho)=(\theta+W(\rho),\rho)$, with $W:=2\pi f/g$ (or $2\pi g/f$). It then follows from this expression that $\rot u$ is integrable and nondegenerate if the twist condition in Definition~\ref{D:IN} is satisfied, which in our case is equivalent to demanding that $|f'g-fg'|\geq \tau>0$ on the whole $\Omega$. It is easy to see that indeed the functions $f_1$ and $f_2$ defining the vector field $u$ can be chosen so that $f$ and $g$ fulfill this twist condition, a concrete example will be presented later.
\end{proof}

We are ready to define the family of steady Euler flows on $\mathbb S^3$ which is used in the statement of Theorems~\ref{T:main0} and~\ref{T:main}.

\begin{definition}\label{D:TI}
We say that a vector field $u\in \SVect^k(\mathbb S^3)$ is a stationary \emph{nondegenerate shear Euler flow} if it has the form~\eqref{E:steady}, and the $C^k$ functions $f_1$ and $f_2$ are such that:
\begin{enumerate}
\item The Bernoulli function $\mathcal B$ of the field $u$, defined by the equation \break $u \times \rot u = \nabla {\mathcal B}$,
 is Morse-Bott\footnote{That is, its critical set is 	 a closed submanifold of $\STR$ and its Hessian is nondegenerate in the normal directions.
 } and its critical set consists of the Hopf link $L_0\cup L_1$.
\item The functions $f$ and $g$, defined in Equation~\eqref{E:rot}, satisfy the twist condition $|f'(\rho)g(\rho)-f(\rho)g'(\rho)|\geq \tau$ for all $\rho\in(0,1)$ and some constant $\tau>0$.
\end{enumerate}
\end{definition}

Note that the second condition implies that $|g(\rho)| + |f(\rho)| \ge \tau_1>0$ for all $\rho$, and that the zero sets  of $f$ and $g$ are finite collections of points. So if $u$ is a steady nondegenerate shear Euler flow, then $\rot u$ does not vanish on $\STR$ and is integrable and nondegenerate (in the sense of Definition~\ref{D:IN})
on $\Omega:=\mathbb S^3\backslash (L_0\cup L_1)$. Besides, since $\min \mathcal B$ and $\max \mathcal B$ are two different critical values,
then one of them, minimum, is attained at $L_0$ and the other, maximum, on $L_1$ (or vice versa). The Hessian in directions transversal to $L_0$ is positive definite, and in directions transversal to $L_1$ is negative definite.

An explicit computation (see Appendix~\ref{App2}) of the Bernoulli function for the field $u$ defined by Equation~\eqref{E:steady} gives the
Bernoulli function expressed via $\rho$ up to an additive constant as
\begin{align}\label{E:Ber}
\mathcal B(\rho)=\int_0^\rho \Big[f_1(s)f_1'(s)+f_2(s)f_2'(s)+4f_1(s)f_2(s)\\+(2s-1)(f_1(s)f_2'(s)+f_2(s)f_1'(s))\Big]ds\,.
\end{align}

Let us now construct an explicit example of a stationary nondegenerate shear Euler flow to show that Definition~\ref{D:TI} is non-empty.

\begin{example}
We set
\begin{equation*}
f_1(\rho)=1+\rho\,,\qquad f_2(\rho)=0\,.
\end{equation*}
Straightforward computations show that the Bernoulli function has the expression
\begin{align*}
\mathcal B(\rho)=\rho+\frac12 \rho^2\,,
\end{align*}
and the functions $f$ and $g$ are
\begin{equation*}
f(\rho)=-4\rho\,,\qquad g(\rho)=2+4\rho\,.
\end{equation*}
Accordingly, since $F$ only assumes values in $[0,1]$, the Bernoulli function $\mathcal B$ is Morse-Bott and its critical set is the Hopf link $L_0\cup L_1$. Moreover, $g$ does not vanish for $\rho\in[0,1]$ and the twist condition holds because
$$
|f'g-fg'|=8\,.
$$
Therefore, the (non-vanishing) vector field
$$
u=(1+\rho)u_1
$$
is a stationary nondegenerate shear Euler flow on $\mathbb S^3$.
\end{example}

\subsection{Nearby stationary Euler flows.}
Here we prove a key result for Theorem~\ref{T:main}, which is interest by itself: a description of all stationary solutions $C^k$-close to steady nondegenerate shear ones.

\begin{theorem}\label{L:main}
Let $u$ be a stationary nondegenerate shear Euler flow in $\mathbb S^3$. There exists a positive constant $\eta$ such
that any stationary solution $\widetilde u$ of the Euler equation in the neighbourhood $B_\eta(u,k)$, $k\geq 7$, has a
Bernoulli function $\widetilde {\mathcal B}$ that is Morse-Bott with a critical set consisting of two closed curves
$\widetilde L_0\cup \widetilde L_1$ isotopic to the Hopf link $L_0\cup L_1$, the isotopy being $C^{k-1}$-close to the identity\,\footnote{
In particular, the distance between $\widetilde L_0\cup \widetilde L_1$  and $L_0\cup L_1$  is $\le C\eta$.
}.
Moreover, the vorticity $\rot \widetilde u$ is integrable and nondegenerate in the domain $\widetilde\Omega:=\mathbb S^3\backslash (\widetilde L_0\cup \widetilde L_1)$, and the integrability functional for the trivial (unknotted) isotopy class of embedded tori in $\mathbb S^3$ assumes the value $\vk_0(\rot \widetilde u)=|\mathbb S^3| =2\pi^2$.
\end{theorem}

\begin{proof}
Take any $\widetilde u\in B_\eta(u,k)$, and let $\widetilde {\mathcal B}$  be its Bernoulli function. Then
\begin{equation*}
\widetilde u \times \rot \widetilde u = \nabla \widetilde {\mathcal B}\,, \qquad \diver \widetilde u=0\,.
\end{equation*}
So $\widetilde {\mathcal B}$ is a $C^k$-smooth function and is a first integral of both $\widetilde u$ and $\rot\widetilde u$. Since
 $\widetilde u\in B_\eta(u,k)$, then
$$
\|\nabla(\mathcal B-\widetilde{\mathcal B})\|_{C^{k-1}}<C\eta\,.
$$
We may assume that $ {\mathcal B}$ and $\widetilde {\mathcal B}$ are normalised by the condition
$ {\mathcal B}(p_0)= \widetilde {\mathcal B}(p_0)=0$, where $p_0$ is a fixed point in $\STR$. Then the estimate above implies that
\begin{equation}\label{Eq:ck}
\|\mathcal B-\widetilde{\mathcal B}\|_{C^{k}}<C\eta\,.
\end{equation}

Since the Hessian matrix $D^2\mathcal B$ is positive definite on the normal bundle of $L_0$ and negative definite on the normal bundle of $L_1$, then by
the implicit function theorem, the critical set of $\widetilde{\mathcal B}$ is contained in a set of two $C^{k-1}$ curves $\widetilde {L_0}\cup \widetilde {L_1}$ isotopic to the Hopf link $L_0\cup L_1$ (the isotopy being $C^{k-1}$ close to the identity):
$$
\text{Cr}(\widetilde{\mathcal B})\subset \widetilde {L_0}\cup \widetilde {L_1}\,, \qquad  \widetilde {L_0}\subset  {L_0}+C\eta, \   \widetilde {L_1}\subset  {L_1}+C\eta,
$$
where $L_j+C\eta$ is the $C\eta$-neighborhood of $L_j$ in $\mathbb S^3$. Moreover,
 $D^2\widetilde{\mathcal B}$ is positive  definite on the normal bundle of $\widetilde {L_0}$ and  negative definite on the normal bundle of
 $\widetilde {L_1}$. We claim that both $\widetilde L_0$ and $\widetilde L_1$ are critical. Indeed, assume that this is not the case, so there is a connected component $\Delta$ of the critical set of $\widetilde{\mathcal B}$ that is a proper subset of, say, $\widetilde L_0$.
Then $\Delta$ is  a point or an interval by the classification of connected one-dimensional closed sets.
Since $\widetilde{\mathcal B}$ is a first integral of the vector field $\rot\widetilde u$,
the critical set of $\widetilde{\mathcal B}$ is invariant under the flow of this field, so the set $\Delta$ is invariant. Moreover, $\rot\widetilde u$ does not vanish on $\mathbb S^3$ because it is $C^{k-1}$-close to the non-vanishing field $\rot u$. This is a contradiction because $\rot\widetilde u$ should vanish at the boundary points of $\Delta$, thus proving that
$$
\text{Cr}(\widetilde{\mathcal B})= \widetilde {L_0}\cup \widetilde {L_1}\,.
$$
It follows from the previous discussion that $\widetilde {\mathcal B}$ is a Morse-Bott first integral of $\widetilde u$ and $\rot \widetilde u$, whose critical set is formed by the two closed curves $\widetilde L_0\cup \widetilde L_1$ that are isotopic to the Hopf link $L_0\cup L_1$, which are periodic orbits of $\rot\widetilde u$. The global minimum and maximum of the function $\widetilde {\mathcal B}$ are $\widetilde L_0$ and $\widetilde L_1$, respectively, and $D^2 {\mathcal B}$  is positive (negative) definite on the corresponding normal
bundle.
Accordingly, the regular level sets
of $\widetilde {\mathcal B}$ are diffeomorphic to two-dimensional tori that are unknotted in $\mathbb S^3$.

Now we prove that $\rot \widetilde u$ is integrable and nondegenerate in $\widetilde\Omega:=\mathbb S^3\backslash (\widetilde L_0\cup \widetilde L_1)$. We first observe that the stationary Euler equation implies that $\widetilde u$ and $\rot \widetilde u$ commute on $\mathbb S^3$,
$$
[\widetilde u, \rot \widetilde u]=0\,,
$$
and that they are linearly independent at each point of $\mathbb S^3\backslash (\widetilde L_0\cup \widetilde L_1)$, see~\cite{AKh}.
Let us first consider two $\eta$-independent neighborhoods $V_0:=\{0\leq \rho<\rho_0\}$ and $V_1:=\{\rho_1<\rho\leq 1\}$
of $L_0$ and $ L_1$, respectively, where $\rho_0$ and $1-\rho_1$ are small enough (recall that the coordinate $\rho$ was
defined in Equation~\eqref{the_coord}). Assume that $\eta$ is so small that $\widetilde{L_j} \subset V_j$.
The proof of the non-symplectic part of the
Arnold-Liouville theorem~\cite[Chapter II, Proposition 1.5]{AKh} shows that there are
 $C^{k-1}$-smooth coordinates $\tilde\rho=\widetilde{\mathcal B}$ and $(\tilde\theta_1, \tilde\theta_2)\in \mathbb T^2$ in the complement
$\mathbb S^3\backslash\{V_0\cup V_1\}$ such that the vector field $\rot \widetilde u$ (and also $\widetilde u$) reads as
$$
\rot\widetilde u={\tilde f} ({\tilde \rho})\partial_{\tilde \theta_1} + {\tilde g}({\tilde \rho}) \partial_{\tilde \theta_2}\,,
$$
for some $C^{k-1}$ functions $\tilde f$ and $\tilde g$. Since the functions $\widetilde{\mathcal B}$ and $\mathcal B$
are $\eta$-close, as well as the fields $(u,\rot u)$ and $(\widetilde u,\rot \widetilde u)$, then in the domain
$\mathbb S^3\backslash\{V_0\cup V_1\}$
the tilde--coordinates can be chosen $\eta$--close
to the original coordinates $(\theta_1,\theta_2,\rho)$ in the  $C^{k-1}$-norm.
 The functions $\tilde f, \tilde g$ are then $\eta$-close to $f$ and $g$. Therefore, since $\rot u$ is integrable and nondegenerate in $\Omega$, we conclude that the functions $\tilde f$ and $\tilde g$ satisfy the twist condition
$$
|\tilde f'\tilde g-\tilde f\tilde g'|\geq \tau-C\eta>0
$$
in $\mathbb S^3\backslash\{V_0\cup V_1\}$, thus showing that $\rot \widetilde u$ is integrable and nondegenerate on that domain.

Now we proceed to prove that $\rot\widetilde u$ is integrable and nondegenerate in $V_0\backslash\widetilde{L_0}$ (the analysis of $V_1$ is completely analogous). Indeed, take the coordinates $(\theta_1,\theta_2,\rho)$ defined in~\eqref{the_coord} to parametrize $V_0$, and assume without loss of generality that $g(\rho)>0$ for $0\leq \rho \leq \rho_0$ ($g$ was defined in~\eqref{E:rot}). Consider the disk
$\Sigma:=\{\theta_2=0\}\cap V_0$. It defines a  Poincar\'e map $\Pi$  for the field $\rot u$, and we see from \eqref{E:rot} that
\begin{equation}\label{pipi}
\Pi(\rho,\theta_1)=(\rho,\theta_1+W(\rho)), \qquad W=2\pi f/g.
\end{equation}
 So $\Pi$ preserves the standard area form
$\mu_2^0 := d\rho \wedge d\theta_1$.  Since $\rot\widetilde u$ is $\eta$-close to $\rot u$, it follows that the disk $\Sigma$ also
defines a  Poincar\'e map $\tilde\Pi$  for the field $\rot\widetilde u$. The latter has a fixed point at $\widetilde{L_0}\cap\Sigma$. Denote
by $D_R \subset \Sigma$ a sufficiently small $\eta$-independent disk, centered at that point. Then
$
\tilde\Pi :D_R \to \Sigma
$
is a $C^{k-1}$-diffeomorphism of $D_R$ onto its image, which is $\eta$-close to $\Pi$. Since $\rot\tilde u$ is divergence-free, then
$\tilde\Pi$ preserves a $C^{k-1}$-area form $\mu_2$, which is $\eta$-close to the standard form $\mu_2^0$, preserved by $\Pi$.
Moreover, the $C^k$ function $\widetilde H:=\widetilde{\mathcal B}|_{\Sigma}$ is a first integral of $\widetilde\Pi$ which has a Morse
minimum at $0\in D_R$.

Moser's lemma implies that for small $\eta$ there exists a $C^{k-1}$-change of variables $\Phi:D_{R_1} \to D_R$,
$R_1:=R/2$, $\eta$-close to the identity, and such that $\Phi(0) =0$,  $\Phi^*\mu_2 =\mu_2^0$.  We will pass to the new
variables, keeping the notations $\tilde\Pi$ and $\tilde H$ for the two introduced objects. Observe that, in the new coordinates, $\tilde H$ is of class $C^{k-1}$, and $\tilde \Pi$ preserves the standard area form $\mu_2^0$.
Next, Corollary~\ref{C:Vey} ensures that there
exists a $C^{k-4}$ change of variables
$
\Psi:D_{R_0} \to  \Psi(D_{R_0} ) \subset D_{R_1}$, $\Psi(0)=0,
$
such that in the polar coordinates in the new variables the transformed diffeomorphism $\hat\Pi :=  \Psi^{-1} \circ \Pi\circ\Psi$ reads as
\begin{equation}\label{map_Vey}
\hat \Pi(r, \theta_1)=(r, \theta_1 + \hat W(r^2))\,,
\end{equation}
where $\hat W(t) \in C^{[(k-4)/2]} \subset C^1$ and $\hat W(t) \in C^{k-4}$ for $t>0$.
Introducing in $D_{R_0} $ the complex coordinate $z=x+iy$ and denoting
$$
f_0:= \hat W(0), \quad \frak z :=e^{i f_0},\quad f_1 :=  \hat W'(0),\quad f(t) :=  \hat W(t) -  f_0\,,
$$
we write the map $\hat\Pi$ as
$$
\hat\Pi (z) = ze^{i  \hat W(|z|^2)} = \frak z z e^{if(|z|^2)} =  \frak z z ( 1+ if_1|z|^2) + o(z^3)\,.
$$
Since $\hat \Pi \in C^{k-4} \subset C^3$, then
$
{\p^3 \hat\Pi}(0)/  {\p z^2 \p\bar z} = 2i\frak z f_1 = 2i\frak z  \hat W'(0).
$
A similar relation holds for $\Pi$, with $W'(0) \ne0$ in view of~\eqref{pipi}. Since $\hat\Pi$ is obtained from $\Pi$ by an $\eta$-small
$C^{k-4}$-perturbation, $k\ge7$,  and by a change of variable which is $\eta$-close to  identity in the $C^{k-4}$-norm, then $\hat W'(0)\ne0$, if
$\eta\ll1$. So $\rot\tilde u$ is integrable and nondegenerate in $V_0\setminus \tilde L_0$ (of course, we assume that $V_0$ and $\eta$ are sufficiently small).

Putting together the previous paragraphs, we conclude that $\rot \widetilde u$ is integrable and nondegenerate in $\widetilde\Omega$. Moreover, since the regular level sets of $\widetilde{\mathcal B}$ are unknotted invariant tori and cover the whole domain $\widetilde\Omega$, we obtain that $\vk_0(\rot \widetilde u)=|\mathbb S^3|=2\pi^2$, and the theorem follows.
\end{proof}

\subsection{Shear stationary flows on the 3-torus}\label{S.4}

Let us first recall the construction of shear stationary solutions on $\mathbb T^3=\{(x,y,z)\in \mathbb R^3 \, (\mod 2\pi)\}$ introduced in~\cite{KKP14}. Consider the divergence-free vector field $u$ defined by
\begin{equation}\label{E:torus}
u = f(z)\partial_x + g(z)\partial_y\,,
\end{equation}
where $f$ and $g$ are $C^k$ $2\pi$-periodic functions. Any function of $z$ is a first integral of $u$, hence the
integral curves of this vector field are tangent to the tori $T_c := \{z = c\}$, and on each torus the field is constant.
The same holds for the vorticity field
$$ \rot u = -g'(z)\partial_x+ f'(z)\partial_y\,.$$
This implies that the field $u$ is a solution of the steady Euler equation on $\mathbb T^3$, since
the fields $u$ and $\rot u$ commute, $[u, \rot u] = 0$.
The corresponding Bernoulli function, which is defined (modulo a constant) by $u\times \rot u=\nabla  \mathcal B$,
is $\mathcal B = (f^2 + g^2)/2$.

The following definition of nondegenerate shear Euler flows on $\mathbb T^3$ is analogous to Definition~\ref{D:TI}, however, unlike the case of $\mathbb S^3$, the fields will be nondegenerate not on a set of full measure, but almost full: for any given $\varepsilon>0$ they will be nondegenerate on a set of measure
$8\pi^3 - \varepsilon$. This turns out to be sufficient for the proof of Theorem~\ref{T:main0}, but does not provide a full description of nearby
steady flows for the validity of an analogue of Theorem~\ref{T:main} for the torus.

\begin{definition}\label{D:torus}
We say that a vector field $u\in \SVect^k(\mathbb T^3)$ is a stationary \emph{nondegenerate shear Euler flow} if it has the form~\eqref{E:torus}, and the $C^k$ functions $f$ and $g$ are such that:
\begin{enumerate}
\item The Bernoulli function $\mathcal B=(f^2(z) + g^2(z))/2$ of the field $u$,
 is Morse-Bott and its critical set consists of a finite number of tori $\{z=z_k~(\mod 2\pi)\}$.
\item The functions $f$ and $g$ satisfy the twist condition $|f''g'-f'g''|\geq \tau>0$ on the set $\mathbb T^3_\tau:=\mathbb T^2\times \Omega_\tau\subset \mathbb T^3$, i.e.  for all $z\in\Omega_\tau\subset [0,2\pi]$, where $\Omega_\tau$ consists of a finite number of disjoint intervals in $[0,2\pi]$ and $\Omega_\tau \to [0,2\pi]\backslash \cup_{j} \tilde z_j$ as $\tau\to 0$. Here $\{\tilde z_j\}$ is the set of points where $f''g'-f'g''$ vanishes, which is assumed to be finite.
\end{enumerate}
\end{definition}

One can easily see that the shear Euler flow $u$ in Equation~\eqref{E:torus} is nondegenerate for generic functions $f$ and $g$. Note also that all the invariant tori of $\rot u$ are ``horizontal" and are of one and the same nontrivial isotopy class.

\begin{remark}\label{L:mainT}
Similarly to Theorem~\ref{L:main}, for a stationary nondegenerate shear Euler flow  $u$ on $\mathbb T^3$ there exists a positive constant $\eta(\tau)$ depending on $\tau$ such that any stationary solution $\widetilde u$ in a neighbourhood $B_\eta(u,k)$, $k\geq 2$, has a Bernoulli function $\widetilde {\mathcal B}$ which is a perturbation of the Morse-Bott function $\mathcal B$ on $\mathbb T^3$. Its level sets on the complement of a neighborhood of $\mathbb T^3\backslash \bigcup_k\{z=z_k\}$ are regular and isotopic to horizontal tori on $\mathbb T^3$. On the other hand, $\widetilde u$ is nondegenerate on a set $\widetilde{\mathbb T}^3_\eta$ diffeomorphic to $\mathbb T^3_{\eta}$, and the integrability functional for the isotopy class  of ``horizontal"  embedded tori in $\widetilde{\mathbb T}^3_{\eta(\tau)}$ tends to
$|\mathbb T^3|=8\pi^3$ as $\tau\to 0$.
\end{remark}

The proof of Remark~\ref{L:mainT} repeats the proof of Theorem~\ref{L:main} for the sphere, except for the description of the critical set of
a perturbation $\widetilde {\mathcal B}$ of the Morse-Bott Bernoulli function ${\mathcal B}$. Now one cannot guarantee that the critical set of $\widetilde {\mathcal B}$ consists of tori, just like for the function ${\mathcal B}$, since the fields $u$ and $\rot u$ become collinear
on those critical sets, which are $2$-dimensional (the fact that the dimension of the critical set of $\mathcal B$ is $1$ is crucial in the proof of Theorem~\ref{L:main}). This is why for the torus one introduces the set $\mathbb T^3\backslash \bigcup_k\{z=z_k\}$ where the Bernoulli function $\mathcal B$ has no critical points.

Nevertheless, Definition~\ref{D:torus} is sufficient to ensure that the inequality $\|\mathcal S_t(v_0)- u_0\|_k>\delta$
in Theorem \ref{T:main0} holds for the field $u_0$ (without the claim about any stationary solution $ \widetilde u_0$). Indeed, given a nondegenerate shear steady field $u_0$ in $\mathbb T^3$, the proof in the next section uses only the existence of a periodic orbit of $\rot u_0$ that is a $(p,q)$-torus knot, which holds for such a field thanks to the nondegeneracy assumption.

\section{Proof of Theorems~\ref{T:main0} and~\ref{T:main} on local non-mixing}\label{S:5}
We will focus on the proof of Theorem~\ref{T:main}, and comment on the (much easier) proof of Theorem~\ref{T:main0} in Remarks~\ref{Rem:tor} and~\ref{Rem:last} below. The proof is divided in two steps. In the first one we show that in any small enough neighborhood of $u_0$, a generic vector field exhibits a positive measure set of ergodic invariant tori that are nontrivially knotted; in the second step we combine this result with the theory developed in~\cite{KKP14}. In the proof, a KAM theorem for generic elliptic points of area preserving maps (Appendix~\ref{app3}), and a suitable suspension construction (Appendix~\ref{app:suspen}) are instrumental. We recall that in this section $k>10$ is a fixed non-integer number, and we shall use the Notation introduced in Section~\ref{s0} without further mention.
\\

\noindent {\bf Step 1:} Let $u_0$ be a stationary nondegenerate shear Euler flow. The main result of this step is the following  proposition:

\begin{proposition}\label{prop:lambda}
For a sufficiently small $\epsilon>0$ and any generic divergence-free vector field
 $v_0$ in  $B_\epsilon(u_0;k)$,
 $\rot v_0$ has a positive measure set of ergodic invariant tori that are not isotopic to the standard (unknotted) torus in $\mathbb S^3$.
 That is,
$
\vk_a(\rot v_0)\geq \lambda(\epsilon)>0\,,
$
for some $a\neq 0$.
\end{proposition}

\begin{remark}\label{Rem:tor}
The same statement holds for  a stationary nondegenerate shear Euler flow  $u_0$ on $\mathbb T^3$
with a full measure set of ``horizontal"  invariant tori:
there exists an arbitrarily close perturbation $v_0$ whose  $\rot v_0$ has a positive measure set of ergodic invariant tori
not isotopic to the standard ``horizontal"   ones on $\mathbb T^3$. We prove the proposition for  $\mathbb S^3$, while the
$\mathbb T^3$ case works verbatim and, more generally, the argument has  a ``local nature".
\end{remark}

\begin{proof}
Consider the vector field $\rot u_0$ in $\Omega=\mathbb S^3\backslash(L_0\cup L_1)$ written in coordinates
$(\theta_1,\theta_2,\rho)$ (see Equations~\eqref{the_coord} and~\eqref{E:rot}), and take an interval
$(a,b)\subset (0,1)$ such that   $g(\rho)$ does not vanish on its closure. We will denote
$
A(a,b) = \mathbb S^1 \times \{\theta_2=0\} \times (a,b),
$
and will identify $A(a,b)$ with the annulus $ \mathbb S^1 \times (a,b)$. The  Poincar\'e return map $\Pi_0$ defined by $\rot u_0$
 at the section $\{\theta_2=0\}$ has the form~\eqref{pipi}; this is
  a twist map on the annulus $A(a,b)$ of class $C^{k-1}$. It is nondegenerate (i.e. $W'(\rho)\neq 0$)
  in view  of the twist condition in the definition of nondegenerate shear Euler flows, and  preserves the area-form $g(\rho)d\theta_1\wedge d\rho$. If $\epsilon\ll1$, then for
  any $v_0 \in B_\epsilon(u_0;k)$, the field $\rot v_0$ also defines a Poincar\'e return map
  \begin{equation} \label{kk}
  \cP_\epsilon^{v_0}: A(a,b) \to A(0,1),
\end{equation}
 belonging to $B_{C\epsilon}(\Pi_0;C^{k-1})$, where $C^{k-1}$ stands for the space of $C^{k-1}$-smooth maps $ A(a,b) \to A(0,1)$. Here and in what follows, $C$ is a constant that may vary from line to line and  does not depend on $\epsilon$.  Note  that  $ \cP_\epsilon^{u_0}=\Pi_0$.
 Note also  that all maps in $B_{C\epsilon}(\Pi_0;C^{k-1})$ are diffeomorphisms on the image if $\epsilon\ll1$, and that for each
 $v_0\in  B_\epsilon(u_0;k)$ the map $\cP_\epsilon^{v_0}$ has an invariant two-form $\mu_2^{v_0}$,  $\epsilon$-close to the
 form $g(\rho)d\theta_1\wedge d\rho$.

Now, let us pick up
a point $c \in (a,b)$ such that $W(c)=2\pi p/q$ for two coprime natural numbers $p,q$. Then any point
 $(\theta_1,c)$ is periodic under  iterations of $\Pi_0$ with period $q$. It corresponds to a periodic orbit
 of $\rot u_0$ in $\mathbb S^3$ that is a $(p,q)$-torus knot over the unknot (the orbits turns $q$ times in the $\theta_2$ direction and $p$ times in the $\theta_1$ direction). The following lemma is instrumental in the proof of Proposition~\ref{prop:lambda}. We recall that an
  elliptic fixed point  of a map  $A(a,b) \to A(0,1)$
  is called {\it
KAM-stable }if it is accumulated by a positive measure set of invariant continuous quasi-periodic curves. We say that an elliptic fixed point
of an area-preserving map $A(a,b) \to A(0,1)$ is {\it nondegenerate} if its eigenvalue $\lambda$ (and its complex conjugate) avoids the resonances $1,2,3,4$ (that is $\lambda,\lambda^2,\lambda^3$ and $\lambda^4$ are not equal to $1$), and the first Birkhoff constant of this fixed point is not zero.

\begin{lemma}\label{L:map}
Let $c\in(a,b)$ be as above and  $\mu_2$ be any $C^{k-1}$-area form on  $\mathbb S^1 \times (0,1)$.
Consider  the set of
$C^{k-1}$ exact diffeomorphisms  $\Pi_\epsilon:  A(a,b) \to A(0,1)
$, preserving  $\mu_2$ and belonging to
$
B_{C\epsilon} (\Pi_0;C^{k-1})
$.
If $\epsilon\ll1$, then  for a generic $\Pi_\epsilon$ the map $\Pi_\epsilon^q$ has a nondegenerate elliptic fixed point which is
KAM-stable and is  $\epsilon$-close to the resonant curve $\mathbb S^1\times \{c\}$.
\end{lemma}

\begin{proof}
We first note that for a generic area-preserving perturbation $\Pi_\epsilon$ that is an exact diffeomorphism (i.e. it satisfies the intersection property), the map $\Pi_\epsilon^q$ has a fixed point that is elliptic (actually, at least $q$ fixed points) and $\epsilon$-close to the resonant curve $\mathbb S^1\times \{c\}$, as a consequence of the Poincar\'e-Birkhoff theorem~\cite[pp. 220-222]{GH}. Now a result of Robinson~\cite[Theorem 9]{Ro70} shows that for a typical $\Pi_\epsilon$ the elliptic fixed point is nondegenerate. We can then apply Moser's stability theorem  (see Appendix~\ref{app3}) to conclude that this
elliptic fixed point of $\Pi_\epsilon^q$ is KAM-stable.  It then gives rise to a KAM-stable elliptic $q$-periodic point for the Poincar\'e
map $\Pi_\epsilon$.
\end{proof}

To prove Proposition~\ref{prop:lambda} we must show that there exists an open and dense set of $v_0 \in B_\epsilon(u_0;k)$ satisfying the desired properties. Once density is established, the openness of the set easily follows from the fact that a nondegenerate KAM-stable elliptic fixed point is robust under $C^k$-small area-preserving perturbations (c.f. Lemma~\ref{L:map}).

To prove the density, the crucial idea is the suspension construction presented in Appendix~\ref{app:suspen}. Indeed, take a vector field $v_1\in B_\epsilon(u_0;k)$ and the Poincar\'e return map $\cP^{v_1}_\epsilon$ of $\rot v_1$. As discussed above, $\cP^{v_1}_\epsilon\in B_{C\epsilon}(\Pi_0;C^{k-1})$ and it preserves an area form $\mu_2^{v_1}$. Moreover, $\cP^{v_1}_\epsilon$ is exact because any divergence-free vector field on $\mathbb S^3$ is exact. For any $\epsilon'>0$, Lemma~\ref{L:map} implies that there exists a map $\Pi\in B_{\epsilon'}(\cP^{v_1}_\epsilon;C^{k-1})$ that preserves $\mu_2^{v_1}$ whose $q$-iterate has a nondegenerate elliptic fixed point which is KAM-stable. Deforming $\Pi$ with an appropriate $\epsilon'$-small homotopy which is the identity in a neighborhood of the resonant curve $\mathbb S^1\times\{c\}$, we can safely assume that $\Pi=\cP^{v_1}_\epsilon$ in the complement of a neighborhood of this curve.

Applying Theorem~\ref{teo:susp} in Appendix~\ref{app:suspen} to the map $\Pi$ we obtain a $C^{k-1}$ divergence-free field $w_0$ in $A\times \mathbb S^1$ whose Poincar\'e map at the section $\{\theta_2=0\}$ is precisely $\Pi$. The fact that $w_0$ has the same regularity as $\Pi$ is a consequence of~\cite[Section~5.2]{Treschev} and the property that the group of exact $C^{p}$ area-preserving diffeomorphisms of the annulus (connected with the identity) is locally connected by $C^{p+1}$ paths provided that $p>1$ is not an integer (see e.g.~\cite{Do82}). Since $\Pi=\cP^{v_1}_\epsilon$ in the complement of a neighborhood of the resonant curve, we also have that $w_0=\rot v_1$ in the complement of a neighborhood of the resonant torus $\mathbb S^1\times\{c\}\times \mathbb S^1$, thus defining a global field on $\mathbb S^3$ that we still denote by $w_0$. Moreover, these fields are close in the sense that $\|w_0-\rot v_1\|_{k-1}<C\epsilon'$, so $w_0$ is $C\epsilon$ close to $\rot u_0$.

Notice that the field $w_0$ has a KAM-stable elliptic periodic orbit $\gamma_\epsilon$. The curve $\gamma_\epsilon$ is a $(p,q)$-torus knot because $w_0$ and $\rot u_0$ are $C^{k-1}$-close and, by construction, this periodic orbit bifurcates as $\epsilon\to 0$ from a degenerate periodic orbit of $\rot u_0$ which is a $(p,q)$-torus knot as well. Being KAM-stable, it is surrounded by ergodic invariant tori of class $C^{k-1}$ which jointly occupy in $\STR$ a set of positive Lebesgue measure. These tori are the boundaries of tubular neighbourhoods of $\gamma_\epsilon$, which is a non-trivial knot. So we conclude that they are not isotopic to the standard (unknotted) torus in $\mathbb S^3$.

Finally, the simply connectedness of $\mathbb S^3$ implies that $w_0$ is the vorticity of a unique divergence-free vector field $v_0$ of class $C^k$,
i.e. $w_0=\rot v_0$,  satisfying the estimate
\[
\|v_0-v_1\|_{k}=\|\rot^{-1}(w_0-\rot v_1)\|_{k}<C\epsilon'\,,
\]
since the linear operator $\rot^{-1}:\SVect^{k-1}(\mathbb S^3)\to \SVect^k(\mathbb S^3)$ is continuous, see Lemma~2.3 in \cite{KKP14}
(recall that $k>10$ is not an integer). The vector field $v_0$ satisfies the conditions in the statement of Proposition~\ref{prop:lambda} and we are done.
\end{proof}

\noindent {\bf Step 2:} To complete the proof of Theorem~\ref{T:main}, fix $\epsilon>0$ and take an isotopy torus of  class $a$ with
the corresponding $\lambda(\epsilon)>0$ as in Proposition~\ref{prop:lambda}.
Since $\vk_a(\rot v_0)\geq \lambda>0$, we have that $\vk_0(\rot v_0)\leq |\mathbb S^3|-\lambda=2\pi^2-\lambda$.
This means that the set of ergodic invariant tori of $\rot v_0$ that are isotopic to the standard torus in $\mathbb S^3$
is not of full measure. It follows that
\begin{equation}\label{E:inv}
\vk_0 (\rot \mathcal S_t(v_0))\leq 2 \pi^2-\lambda
\end{equation}
for all $t$ for which the Euler flow is defined, since $\vk_0$ is preserved by the flow, cf. Theorem~\ref{t3}.

Now, take any stationary Euler flow $\widetilde u_0\in B_\eta(u_0;k)$, $k>10$ non-integer, for $\eta>0$ as described in Theorem~\ref{L:main}. The vorticity $\rot \widetilde u_0$ of any such flow is integrable and nondegenerate and it has
a full measure of standard invariant tori:
$$
\vk_0(\rot \widetilde u_0)=2\pi^2\,.
$$


By Theorem 4.7 from~\cite{KKP14} if $\| \tilde u_0 -v_1\|_{C^k} = \delta_1$ for a non-integer $k>10$, then
\begin{equation}\label{E:estim}
\vk_0(\rot v_1)\ge 2\pi^2 -C\delta_1^2\,.
\end{equation}
The regularity $k>10$ is necessary because we have to apply Herman's theorem~\cite{He83} to the map $\hat\Pi$ in Equation~\eqref{map_Vey} whose twist term $\hat W(\rho)$ is of class $C^{[(k-4)/2]}$. Therefore, $[\frac{k-4}{2}]>3$ if and only if $k>10$.

From the estimates~\eqref{E:estim} and~\eqref{E:inv} we conclude that
$$
\|\mathcal S_t(v_0)-\widetilde u_0\|_{k}>C\delta_1=:\delta
$$
for all $t$ for which the Euler flow is defined. This completes the proof.

\begin{remark}\label{Rem:last}
For the proof of Theorem~\ref{T:main0}   we directly apply~\cite[Theorem~4.7]{KKP14} to the evolution $\mathcal S_t(v_0)$, where $v_0\in B_\epsilon(u_0;k)$ is the generic field in Remark~\ref{Rem:tor}. Note that the set $\mathbb T^3_\tau$ where $u_0$ is nondegenerate has full measure as $\tau\to 0$, which is enough to apply the aforementioned result from~\cite{KKP14}. In fact, since for Theorem~\ref{T:main0} we do not perturb the steady state $u_0$, we do not need to analyze the critical set of the Bernoulli function $\mathcal B$ and hence we do not use Vey's theorem from Appendix~\ref{app4} (which is the main source for the lost of regularity in the proof of Theorem~\ref{T:main}), so the regularity for which Theorem~\ref{T:main0} holds is just $k>5$.
\end{remark}

\section*{Acknowledgements}
We are  indebted to H.~Eliasson, D.~Treschev and D.~Turaev for fruitful discussions.
B.K. was partially supported by an NSERC research grant. B.K. also thanks the ICMAT in Madrid for kind hospitality during his visit.
S.K. was supported by the grant 18-11-00032 of the Russian Science Foundation.
D.P.-S. was supported by the grants MTM2016-76702-P (MINECO/FEDER) and Europa Excelencia EUR2019-103821 (MCIU), and partially supported by the ICMAT--Severo Ochoa grant SEV-2015-0554.

\appendix

\section{The Bernoulli formulation of the stationary Euler equation}\label{app1}

In Euclidean space, it is standard (see e.g.~\cite[Chapter II.1]{AKh}) that the stationary Euler equation
$\nabla_u u =-\nabla p$ can be written in the following equivalent way:
\begin{equation} \label{eq:euler}
u\times \omega=\nabla \mathcal B\,, \;\; \diver u=0\,,
\end{equation}
where $\mathcal B:= p+\frac12|u|^2$ is the Bernoulli function. The same formulation holds for steady fluid flows on a general Riemannian $3$-manifold $(M,g)$. This fact is well known to experts, but difficult to find in the literature, so let us provide a proof.

First, we recall the definition of vector product in a Riemannian $3$-manifold. If $X$ and $Y$ are two vector fields on $M$, its vector product $X\times Y$ is the unique vector field defined as
$$
i_{X\times Y}g=i_Yi_X\mu\,,
$$
where $\mu$ is the volume form on $M$, and $i_Zg$ denotes the 1-form $Z^\flat$ dual to the vector field $Z$ by means of
 the metric $g$.

The vorticity $\omega=\rot u$ of a field $u$ is defined as the only vector field on $M$  satisfying
$$
i_\omega \mu=d\alpha\,,
$$
where the 1-form $\alpha:=u^\flat$ is metric-dual to the vector field  $u$.

The first observation to prove Equation~\eqref{eq:euler} is that the 1-form $\left(\nabla_u u\right)^\flat$ dual to the covariant derivative $\nabla_u u$ is~\cite[Chapter IV.1.D]{AKh}:
$$
\left(\nabla_u u\right)^\flat=L_u\alpha - \frac12 d(\alpha(u))\,.
$$
Using Cartan's formula for the Lie derivative, this differential form can also be written as
$$
i_ud\alpha+\frac12 d(\alpha(u))=i_ui_\omega\mu +\frac12 d(\alpha(u))\,,
$$
where we have used the definition of the vorticity to write the second expression. Accordingly, the stationary Euler equation $\nabla_u u=-\nabla p$ reads in terms of differential forms as:
$$
i_ui_\omega\mu +\frac12 d(\alpha(u))=-dp\,\,\,\,\,\Leftrightarrow\,\,\,\,\,i_\omega i_u\mu = d(p+\frac12 \alpha(u))\,.
$$
Using the definition of the vector product on manifolds, and setting $\mathcal B:=p+\frac12\alpha(u)$ as the Bernoulli function, the (dual) vector formulation of the stationary equation reads
$$
u\times \omega=\nabla \mathcal B\,,
$$
as required.

\section{Some computations for the shear steady states}\label{App2}
In this section we shall use the notations introduced in Section~\ref{S.3} without further mention. The dual 1-forms (using the Euclidean metric) of the vector fields $u_1$ and $u_2$ are
\begin{equation*}
\alpha_1=-ydx+xdy+\xi dz -zd\xi \qquad \alpha_2=-ydx+xdy-\xi dz +zd\xi\,.
\end{equation*}
It is easy to check that in terms of the coordinates $(\theta_1,\theta_2,\rho)$ on $\mathbb S^3$, these forms read as
\begin{equation*}
\alpha_1=\rho d\theta_1-(1-\rho)d\theta_2\qquad \alpha_2=\rho d\theta_1+(1-\rho)d\theta_2\,,
\end{equation*}
and their exterior derivatives are given by
\begin{equation*}
d\alpha_1=-d\theta_1\wedge d\rho-d\theta_2\wedge d\rho \qquad d\alpha_2=-d\theta_1\wedge d\rho+d\theta_2\wedge d\rho\,.
\end{equation*}
Recalling that the volume form in these coordinates is $\mu=\frac12 d\theta_1\wedge d\theta_2\wedge d\rho$, these expressions and the fact that $i_{\rot u_i}\mu = d\alpha_i$ imply that $\rot u_1=-2u_1$ and $\rot u_2=2u_2$.

Now we are ready to compute the rot of the vector field $u=f_1(\rho)u_1+f_2(\rho)u_2$. Its dual 1-form and exterior derivative are
\begin{align*}
\alpha&=\rho(f_1+f_2)d\theta_1+(1-\rho)(-f_1+f_2)d\theta_2\,,\\
d\alpha&=\Big(-\rho f'_1-\rho f'_2-f_1-f_2\Big)d\theta_1\wedge d\rho \\& +\Big((1-\rho)f'_1-(1-\rho)f'_2-f_1+f_2\Big)d\theta_2\wedge d\rho\,.
\end{align*}
Since $\rot u$ satisfies that $i_{\rot u}\mu=d\alpha$, a straightforward computation yields
\begin{align*}
\rot u&=\Big((2\rho-1)(f'_2-f'_1)+2(f_2-f_1)+f'_1-f'_2\Big)\partial_{\theta_1}\\&+\Big(2\rho(f'_1+f'_2)+2(f_1+f_2)\Big)\partial_{\theta_2}\,,
\end{align*}
which is equal to the expression given in Section~\ref{S.3}.

To compute the Bernoulli function $\mathcal B$, we simply use the fact (see Appendix~\ref{app1}) that
$$
d\mathcal B=-i_ud\alpha\,,
$$
which, after a few computations, gives
$$
d\mathcal B=\Big(f_1f_1+f_2f_2'+4f_1f_2+(2\rho-1)(f_1f_2'+f_2f_1')\Big)d\rho\,.
$$
Integrating this expression, we obtain that $\mathcal B$ is a function of $\rho$  given by the formula \eqref{E:Ber}.

\section{KAM stability of generic elliptic points}\label{app3}

The goal of this appendix is to show that a generic elliptic point (in the sense of Lemma~\ref{L:map}) of a $C^4$ area-preserving diffeomorphism of a disk is accumulated by continuous invariant curves with dense orbits, jointly forming a set of positive measure.
This fact was stated  by J.~Moser~\cite[Theorem 2.12]{Moser} with an idea of the proof given, its full implementation is apparently new.
We learned the implementation of this idea, presented below, from D.~Turaev, to whom we are grateful for the explanation.

Let $\Pi$ be a $C^4$ area-preserving diffeomorphism of a disk with an elliptic fixed point at the origin.
 Using the complex notation $z\in \mathbb C$, it can be written in the form
\begin{equation*}\label{map1}
\Pi(z)=e^{i\omega}z+H(z,\overline z)\,, \qquad z\in D_{r_0}\,,
\end{equation*}
where $H$ is a complex-valued function of class $C^4$, $\omega\in\mathbb R$ is the rotation number of the elliptic point $z=0$, and $D_{r_0}:=\{z: |z|<r_0\}$. We assume that the
map is non-resonant in the sense that
\begin{equation}\label{cond}
\text{
$ e^{i\omega k}  \ne1\ $ for $\ k=1,2,3,4$\,. }
\end{equation}
Then, Birkhoff's normal form theorem~\cite[Theorem 2.12]{Moser} ensures that if $r_0$ is small enough,
there exists an analytic area-preserving change of coordinates taking $\Pi$ to the form
\begin{equation*}
\Pi(z)=e^{i(\omega+\alpha|z|^2)}z+\widetilde H(z,\overline z)\,,
\end{equation*}
where $\widetilde H$ is a $C^4$ function such that $\partial_{z,\overline z}^\beta \widetilde H(0,0)=0$ for $|\beta|\leq 3$, which implies that
\begin{equation}\label{hi}
\text{
$|\partial_{z,\overline z}^\beta \widetilde H|\leq C|z|^{4-|\beta|}\ $ in $\ D_{r_0},\ $
 if $\ |\beta|\leq 4$\,. }
\end{equation}
 We assume
 that the first Birkhoff constant  $\alpha $ does not vanish; this is a  generic condition.

Passing to the polar coordinates $z=re^{i\varphi}$, we re-write $\Pi$ as
\begin{equation*}
\Pi_r(r,\varphi)=r+f(r,\varphi)\,, \qquad \Pi_{\varphi}(r,\varphi)=\varphi+\omega+\alpha r^2+ {g(r,\varphi)} \,,
\end{equation*}
where $f$ and $g$ are $C^4$ functions in the punctured disk $ D_{r_0}\backslash \{0\}$. Rescaling  the radius as $r=\epsilon R$, $\epsilon\ll1$,
we  write $\Pi$ in the form
\begin{equation}\label{map2}
\Pi_R(R,\varphi)=R+F(R,\varphi)\,, \quad \Pi_{\varphi}(R,\varphi)=\varphi+\omega+\alpha \epsilon^2 R^2+G(R,\varphi)\,,
\end{equation}
where
$$
(R, \varphi) \in
\mathbb A:=\{1<R<2, \varphi\in [0,2\pi] \}\,,
$$
and $ F(R,\varphi)= \epsilon^{-1} f(\epsilon R, \varphi)$, $ G(R, \varphi) = g(\epsilon R, \varphi)$.

Now we will estimate the functions $G$ and $F$.
To estimate $G$, we define the auxiliary function
$$
\Phi(R, \varphi) := {\widetilde H(z,\overline z)}{z^{-1}} e^{-i(\omega+\alpha|z|^2)}\;,\text{ with} \; z= \epsilon Re^{i\varphi}\,,
$$
so that $G(R,\varphi)=\text{Im}\ln\Big(1+\Phi(R,  \varphi )\Big)$. If $(R, \varphi) \in \mathbb A$, then $\epsilon<r<2\epsilon$.
So by~\eqref{hi}  with  $\beta=0$, for small $\epsilon$ we have that
$|\Phi| <1/2$,  and  accordingly $\ln\big(1+\Phi(R,\varphi) \big)$ is a $C^4$ function of $(R, \varphi)$ in $\mathbb A$. Noticing that
$\p^k/\p R^k = \epsilon^k\p^k/\p r^k$, we derive from \eqref{hi} the estimates
\begin{equation}\label{new_hi}
\Big| \frac{\p^k}{\p R^k} {\widetilde H(R, \varphi)}    \Big| =\epsilon^k \Big| \frac{\p^k}{\p r^k} {\widetilde H}    \Big| \le C\epsilon^4,\quad
\Big| \frac{\p^k}{\p \varphi^k} {\widetilde H(R, \varphi)}    \Big| \le C\epsilon^4\,,
\end{equation}
for $0\le k\le4$ and $(R,\varphi)\in \mathbb A$. It then follows that
 $
 \partial^\beta_{R, \varphi}  \ln\big(1+\Phi(R, \varphi) \big) =O(\epsilon^3)$ for $ (R,\varphi) \in \mathbb A$ and $ |\beta|\le4$.
 Therefore
$\|G\|_{C^{4}(\mathbb A)}\leq C\epsilon^3\,.$

To estimate $F$ we consider the function
$
(\Pi_R)^2= R^2 +2RF +F^2 =: R^2+ J(R, \varphi).
$
Then $F= (\Pi_R^2)^{1/2} -R = R\big( \sqrt{1+R^{-2} J(R,\varphi})-1\big) $.  Since
$\epsilon^2\Pi_R^2(R,\varphi)  = |\Pi(z)|^2$ with $z= \epsilon Re^{i\varphi}$, then
$$
J(R,\varphi) = 2\epsilon^{-1} \text{Re} \big( e^{-i(\omega +\epsilon^2 \alpha R^2) }R e^{-i\varphi}  {\widetilde H(R, \varphi)}  \big) +\epsilon^{-2}
 {\widetilde H}  \overline {\widetilde H}(R, \varphi)\,,
$$
and in view of the estimates~\eqref{new_hi},  $\|J\|_{C^{4}(\mathbb A)}<C\epsilon^3$. We then conclude that $\|F \|_{C^{4}(\mathbb A)}<C\epsilon^3$.

To study the invariant curves  of the map~\eqref{map2}, we consider its $N$-th iterates with $N\le \epsilon^{-2}$. Writing the map  $\Pi^N$ as
\begin{equation*}
\Pi^N_R(R,\varphi)=R+F_N(R,\varphi)\,, \quad \Pi^N_{\varphi}(R,\varphi)=\varphi+\omega_N+\alpha N\epsilon^2 R^2+G_N(R,\varphi)\,,
\end{equation*}
where $\omega_N:=N\omega\, (\text{mod }2\pi)$,
we check by induction that  $F_N$ and $G_N$ are bounded as $\|F_N\|_{C^4(\mathbb A)}<C\epsilon$ and $\|G_N\|_{C^{4}(\mathbb A)}<C\epsilon,$ for
$1\le N \le  \epsilon^{-2}$.

Denote by $N_\epsilon $
the integer part of  $ \epsilon^{-2}$. Then $ \alpha N_\epsilon \epsilon^2 R^2 = \alpha R^2 +O(\epsilon^2)$.
Next let us choose  a sequence $\epsilon_j\to 0$
such that $\omega_{N_j}\to \Omega \,  (\text{mod }2\pi)$, where $N_j$ stands for $N_{\epsilon_j}$. We set
$\
\delta_j = \max( |\omega_{N_j} - \Omega|, \epsilon_j),
$
and consider the twist-map
$$
\widetilde \Pi(R,\varphi)=(R,\varphi+\Omega+ \alpha R^2), \qquad \alpha\ne0.
$$
Then
$\
\|\Pi^{N_j}-\widetilde \Pi\|_{C^{4}(\mathbb A)}<C\delta_j\,.
$
Now Herman's twist theorem~\cite{He83} implies that for $j$ large enough the area--preserving map
 $\Pi^{N_j}$ of  the annulus $\mathbb A$
  has a positive measure set of continuous
   quasi-periodic invariant curves with Diophantine rotation numbers.

These curves also are invariant for the map $\Pi$. This follows from the uniqueness of Diophantine invariant curves in Herman's twist theorem~\cite[Section 5.10]{He83} (if $\gamma$ is a Diophantine invariant curve of $\Pi^{N_j}$, since $\Pi(\gamma)$ is also an invariant curve of $\Pi^{N_j}$ with the same rotation number, then $\Pi(\gamma)=\gamma$). However let us provide a self-contained proof of this fact for the sake of completeness. Indeed, let $\gamma$ be one of the curves, invariant for
 $\Pi^{N_j}$, and let $\gamma'$ be this curve, written in the $z$-variable. It is
  invariant for the original diffeomorphism $\Pi(z)$.  To prove this, denote by $D'$ the domain bounded by $\gamma'$. Since the map $\Pi(z)$
   is area-preserving, then the boundary of $\Pi(D')$
 intersects $\gamma'$, i.e. $\Pi (\gamma')$ intersects $\gamma'$, and $\Pi (\gamma)$ intersects $\gamma$.
  Take any $p\in \gamma$ such that $\Pi(p)\in \gamma$.
  Considering its  orbit under the mapping $\Pi^{N_j}$,
  $\Gamma:=\cup_{m\in \mathbb Z}\Pi^{mN_j}(p)\subset \gamma$, we observe that
   $\Pi(\Gamma)\subset \gamma$. Indeed, denoting by $\gamma_m:=\Pi^{mN_j}(p)$ and $\gamma_m':=\Pi(\gamma_m)$ , we have that $\gamma_0'=\Pi(p)\in\gamma$, $\gamma_1'=\Pi^{N_j+1}(p)=\Pi^{N_j}(\gamma_0')\in\gamma$, etc.
     Since $\Gamma$ is dense in $\gamma$ (it is a quasi-periodic trajectory of the map $\Pi^{N_j}$),
  we conclude that $\Pi(\gamma)\subset \gamma$, and in fact $\Pi(\gamma)=\gamma$ because $\Pi$ is a diffeomorphism, thus showing that $\gamma$ is an invariant curve of $\Pi$. Clearly an orbit of each point in $\gamma$ is dense in $\gamma$.

Summarising, we conclude that the
map $\Pi(z)$  has a positive measure set of continuous
  invariant curves with dense orbits; written in the complex Birkhoff's coordinate $z$ they lie
 in the annulus $\{\epsilon_j<|z|<2\epsilon_j\}$.
 Since the above argument can be repeated for infinitely many values of the small parameter $\epsilon_j\to0$, we deduce that :\\
 {\it If
 the area-preserving map $\Pi(z)$ satisfies~\eqref{cond} and its first Birkhoff constant $\alpha\ne0$, then
 $\Pi$ exhibits a positive measure set of continuous invariant curves with dense trajectories,
  contained in a sequence of disjoint annuli accumulating at the origin, so the KAM-stability of the elliptic fixed point follows.}

\section{The Vey theorem in $\R^2$}\label{app4}

Consider the plane $\R^2=\{ x=(x_1,x_2)\}$, endowed with the standard area-form $\om_0= dx_1\wedge dx_2$.
 By $D_\rho$ we denote the disc $\{|x|<\rho\},\ \rho>0$, and we define the action variable $I:=r^2/2$, where $(r,\phi)$ denote the polar coordinates in the plane.

\begin{theorem}\label{t_Vey}
Let $H\in C^k(D_\rho), \  k\ge4$, such that $H(0)=0$, $dH(0)=0$ and $d^2 H(0) >0$. Then there exists a $C^{k-3}$ area-preserving diffeomorphism
$\Psi^+: D_\rho \to D_{\rho'}$, $\Psi^+(0) =0$, and a $C^{[k/2]-1}$ function $h$, $h(0)=0$, $h'(0)\ne0$, such that
$H(x) = h\big( I(\Psi^+(x))\big)$.
\end{theorem}

An analytic version of this result is due to Vey and is well known~\cite{Vey}.
We were not able to find in the literature a finite-smoothness
version of this theorem and instead give below its proof. It is based on ideas of Eliasson's work~\cite{El}, where a much
more complicated multi-dimensional version of this result is established without explicit control on how the smoothness
of $\Psi^+$ depends on $k$. More precisely, we follow the interpretation of Eliasson's proof, given in~\cite[pp. 10-15]{KP}, skipping the
infinite-dimensional technicalities. Our notation mostly agree with that of~\cite{KP}, and we refer there for missing details.

\begin{proof}Multiplying $H$ by a positive constant and making a linear area-preserving change of coordinates we
 can safely assume that $D^2H(0) = \tfrac12$Id. Accordingly, all changes of variables we are performing below are of the form $x\mapsto x+O(x^2)$.
 By ``a germ" we mean ``a germ at the origin" of a function, or of a vector-field, etc.
\\

\noindent{\it Step 1} (Morse lemma). Applying Morse lemma we find a germ of diffeomorphism\,\footnote{The
 smoothness of  $\Psi$  follows from a more general result in \cite{Tak}.} of class $C^{k-1}$
  $\Psi$  such that
 $$
 i) \qquad H(x) = \tfrac12 |\Psi(x)|^2, \qquad \Psi (x) = x+ O(x^2)\,.   \qquad  \qquad  \qquad  \qquad  \qquad
 $$
 We denote by $G$ the germ of $\Psi^{-1}$ and set
 $$
 \om_1 :=G^* \om_0, \; \om_\Delta := \om_1 - \om_0,\; \al_0:= \tfrac12( x_1dx_2 - x_2dx_1),\; \al_1 := G^*\al_0, \; \al_\Delta := \al_1-\al_0\,.
 $$
Then $d\al_j = \om_j$ and $d\al_\Delta = \om_\Delta$. Write $\al_\Delta $ as $W(x) dx$; then $ W= O(x^2)$. Clearly all introduced objects
are $C^{k-2}$-smooth.

We denote by $(\cdot, \cdot)$ the standard scalar product in $\R^2$, and by $\lan\cdot, \cdot\ran$ the usual pairing between 1-forms and vector fields, and
write 2-forms in $\R^2$ as $(J(x) dx, dx)$, where $J(x)$ is an antisymmetric operator in $\R^2$ (i.e., a $2\times 2$ antisymmetric matrix), and
 $
 (J(x) dx, dx)(\xi,\eta) = (J(x) \xi, \eta)
 $
 for any two vector fields $\xi,\eta$. Then $\om_0 = (J_0 dx, dx)$, where $J_0(x_1,x_2) = (-x_2/2, x_1/2)$,
  and $\om_1(x) = (\bar J(x) dx, dx)$ with
 $\bar J(0) = J_0$.
\\

 \noindent{\it Step 2} (Averaging in angle). Denote by $\Phi_\theta, \ \theta\in\R$, the operator of rotation by angle $\theta$,
 $\Phi_\theta(r,\phi) := (r, \phi+\theta)$. Then $\Phi^*_\theta \al_0 =\al_0$,  $\Phi^*_\theta \om_0 =\om_0$ and
 $$
 \Phi^*_\theta  (V(x)dx) = \Phi_{-\theta}  V(\Phi_\theta x) dx, \quad  \Phi^*_\theta  (J(x)dx, dx) =  (\Phi_{-\theta}  J(\Phi_\theta x)\Phi_\theta dx, dx).
 $$
 It is easy to see that $\Phi_{-\theta}  J(\Phi_\theta x)\Phi_\theta=J(\Phi_\theta x)$. Now consider the averaging operator $M$, where for a function $f$,  $Mf(x) :=\intt f(\Phi_tx)\,dt$, while for a form $\beta$,
 $M\beta (x) :=\intt  (\Phi_t^*\beta)(x)\,dt$. Then
 $$
 M(J(x) dx, dx) =  (MJ(x) dx, dx),\qquad  MJ(x) =\intt \Phi_{-\theta} J(\Phi_\theta x) \Phi_\theta\,d\theta\,.
 $$
 Accordingly, the form $M\om_1$ is $(M\bar J(x) dx, dx)$. Writing the  operator $\bar J(x)$ as
 $$
  \bar J(x) = J_0 +(x, \nabla_x \bar J(0)) + J_2(x), \qquad J_2 =O(x^2)\,,
 $$
 we easily see that $M\bar J = J_0 + MJ_2$, where $C^{k-2}\ni MJ_2(x) =O(x^2)$. Now consider the linear homotopy of $J_0$ and  $M\bar J$:
 $$
 (M\bar J)^\tau := (1-\tau)J_0 +\tau M\bar J  = J_0 + \tau MJ_2\,,
 $$
with $\tau\in [0,1]$, and set $\hat J^\tau(x) := - \big ((M\bar J)^\tau(x)\big)^{-1}= -\big( J_0 +\tau M J_2(x)\big)^{-1}$. This is a germ of a
operator-valued map (an antisymmetric $2\times 2$ matrix) of class $C^{k-2}$.

 Let us return to the $1$-form $\al_\Delta =W(x) dx$, consider $MW(x)$, which is given by
 $$
 MW(x)=\intt \Phi_{-\theta}W(\Phi_\theta x)\,d\theta\,,
 $$
 and define the $C^{k-2}$ non-autonomous vector-field
 $
 V^\tau(x) := \hat J^\tau(x) (MW)(x)$; obviously  $V^\tau = O(x^2)\,.
 $
Consider the differential equation
$$
\dot x(\tau) = V^\tau(x(\tau))\,,
$$
and denote by $\vp_\tau$, $0\le \tau\le1$, the germ of its flow-map. Then $\vp_0=$id, $\vp_\tau(x) = x+ O(x^2)$ and all germs $\vp_\tau$
commute with rotations because the vector field $V^\tau$ does. By direct calculation, using Cartan's formula and that $d (M\alpha_\Delta)=M\omega_1-\omega_0$, we verify that
$
\varphi_\tau^* \hat\om^\tau= \const,$ where  $ \hat\om^\tau = \big( (M\bar J(x))^\tau dx, dx\big).
$
Therefore
\begin{equation}\label{d3}
\vp_1^* \hat\om^1 = \vp_1^* M\om_1 =
 \hat\om^0 = \om_0\,.
\end{equation}
Now let us set $ \bar\Psi :=\vp_1^{-1} \circ \Psi$. This is a germ of a $C^{k-2}$ diffeomorphism, satisfying $ \bar\Psi (x) = x+O(x^2)$. Since
$\vp_1^{-1}$ commutes with rotations, then $2I( \vp_1^{-1}(y)) = \tilde h(2I(y))$. To see which kind of function $\tilde h$ is, let us denote
$2I= r^2$ and take $y=(r,0)$. Then $2I(y) = r^2$, so $ \tilde h(r^2) = 2I( \vp_1^{-1}(r,0)) =: f(r)$. The function $f$ is of class $C^{k-2}$ and
even, so by Whitney's theorem~\cite{Whi} $\tilde h\in C^{[k/2] -1}$. Since $\vp_1^{-1} (r,0) = (r,0) +O(r^2)$, then $\tilde h(2I) = 2I +o(I)$,
and $\tilde h'(0) =1$. So relation $i)$ implies that
$$
 i') \qquad H(x) = h(I(\bar \Psi(x)), \qquad h(0)=0,\ h'(0) =1, \ h\in  C^{[k/2] -1} \,. \qquad  \qquad  \qquad
 $$

Now we re-denote $\bar\Psi$ to $\Psi$. We have arrived at the same situation as in Step 1, but with $\Psi\in C^{k-2}$, \ $i)$ is
 replaced by $i')$ and, in view of \eqref{d3},
 \begin{equation}\label{d33}
M\om_1 = \om_0\,,
\end{equation}
where we recall that $\omega_1=(\Psi^{-1})^*\omega_0$.
\\

\noindent{\it Step 3} (End of the proof). By~~\eqref{d33}, $dM\al_\Delta = M d\al_\Delta = M( \om_1 -\om_0) =0$. So $M\al_\Delta = dg$, $g(0)=0$,
 where $g$ is a $C^{k-1}$-germ. Since $dMg = Mdg = M\al_\Delta$, then by replacing $g$ with $Mg$ we achieve that $dg =M\al_\Delta$ and $Mg=g$.
 Accordingly, $\nabla g(0)=0$ and $g(x) = O(x^2)$. Denote by $\chi$ the vector field of rotations $-x_2 \partial_{x_1} + x_1\partial_{x_2}$. Since $g$ is
 rotationally invariant, then $\lan dg, \chi\ran =0$. So
  \begin{equation}\label{d2}
M\lan \al_\Delta, \chi\ran = \lan M\al_\Delta, \chi\ran= \lan dg, \chi\ran =0\,.
\end{equation}
We set $T(x) = \lan \al_\Delta, \chi\ran$.  This is the germ of a $C^{k-2}$-function, satisfying $MT=0$. Since $\al_\Delta = O(x^2)$ and
$\chi = O(x)$, then $T=O(x^3)$. Let us consider the differential equation for a germ of a function $f$:
\begin{equation}\label{d4}
\chi(f) = T(x).
\end{equation}
Since $MT=0$, it is easy to solve it in polar coordinates for a germ $f$, satisfying $Mf=0$:
$
f(r,\phi) = \intt t \, T(r, \phi+t)\,dt.
$
In Cartesian coordinates the solution $f$ reads as
$$
f(x) = \intt t \, T(\Phi_t(x))\,dt\,.
$$
Similar to $T$, $f\in C^{k-2}$ and $f=O(x^3)$.

Recalling that $(\Psi^*)^{-1} \om_0 =: \om_1 = (\bar J(x) dx, dx)$, where $ \bar J\in C^{k-2}$ with $\bar J(0) = J_0$, we interpolate $J_0$ and $\bar J$
by setting
$
 \bar J^\tau := (1-\tau) J_0 + \tau \bar J,
$
and define  $ J^\tau (x) := -(\bar J^\tau (x))^{-1}$. Then $\bar J^\tau$ and $J^\tau$ are germs of antisymmetric operators of class $C^{k-2}$.
Denote
$$
\om^\tau := (1-\tau) \om_0 + \tau \om_1 = ( \bar J^\tau (x) dx, dx)\,,
$$
and set
$
V^\tau(x) := J^\tau (x) (W(x) -\nabla f(x)).
$
Then $C^{k-3} \ni V^\tau = O(x^2)$. Consider the ODE
 \begin{equation}\label{d5}
\dot x = V^\tau (x), \quad 0\le\tau\le1\,,
\end{equation}
and denote by $\vp_\tau$ the germ of its flow-maps. Then $\vp_\tau(x) = x+O(x^2)$, and another simple calculation shows that
$
\vp_\tau^* \om^\tau = \const.
$
So $\vp_1^* \om_1 = \om_0$. That is, the germ of the diffeomorphism
$$
 \Psi^+ := \vp_1^{-1} \circ \Psi \in C^{k-3}, \qquad \Psi^+(x) = x+O(x^2),
$$
satisfies $(\Psi^+)^*\om_0 =\om_0$, i.e. $\Psi^+$ is an area-preserving diffeomorphism.

Finally notice that
 \begin{equation*}
 \begin{split}
\om^\tau(V^\tau(x), J_0x) = ( \bar J^\tau(x) V^\tau(x), J_0x) = -(W(x) -\nabla f(x), J_0x) \\
= - \frac12\lan \al_\Delta, \chi\ran + \frac12\lan df, \chi\ran =0
\end{split}
\end{equation*}
by Equation~\eqref{d4} (where $T=  \lan \al_\Delta, \chi\ran $). Since $\om^\tau(V^\tau, V^\tau)=0$, then $V^\tau(x) \parallel J_0x$, i.e. the vector field $V^\tau$ is tangent to the foliation defined by $\chi$. Therefore, the
solutions of Equation~\eqref{d5} satisfy
$
 (d/ d\tau) |x(\tau)|^2 = 2 (V^\tau(x), x) =0.
$
That is $ |\vp_\tau(x)|^2= |x|^2$ for all $\tau$. Then, $I(\vp_\tau(x)) =I(x)$, so the germ $\Psi^+$ still satisfies $i')$, which completes the proof of the theorem.
\end{proof}

The theorem above implies Vey's theorem for local area preserving diffeomorphims of class $C^k$ in $\R^2$:

\begin{corollary}\label{C:Vey}
Let $\Pi$ be a $C^k$ area preserving
diffeomorphism of $D_{\rho_1}$ onto its image, $\Pi(0)=0$.
Assume that it admits a first integral $H\in C^k(D_{\rho_1})$ (i.e. $H\circ \Pi=H$) such that $H(0)=0$, $dH(0)=0$ and $d^2 H(0) >0$.
Then there exists $\rho_0>0$ and an area preserving
 $C^{k-3}$-smooth change of variables
$\Psi: D_{\rho_0} \to \Psi ( D_{\rho_0} ) \subset  D_{\rho_1} $,
such that
$\Psi(0)=0$ and the transformed diffeomorphism $\hat\Pi = \Psi^{-1}\circ\Pi\circ\Psi$ in polar coordinates reads as
$
\hat\Pi (r,\phi) =(r, \phi+ \hat W(r^2)),
$
for some  $C^{[(k-3)/2]}$ function $\hat W$.
\end{corollary}
\begin{proof}
By Theorem~\ref{t_Vey} there exists an area preserving local $C^{k-3}$--diffeomorphism $\Psi$ such that $\Psi(0)=0$ and
$H\circ\Psi = \hat H(x)$, where
$\hat H(x) =h(x_1^2+ x_2^2)$ for some  $C^{[k/2]-1}$ function $h$. Since $H$ is a first integral of $\Pi$, then the transformed map
$\hat\Pi = \Psi^{-1}\circ\Pi\circ\Psi$ in polar coordinates reads
 as $\hat\Pi(r,\phi)=(r,\phi+V(r,\phi))$. The fact that $\hat\Pi$ preserves the standard area form $rdr\wedge d\phi$ implies that $V(r,\phi)\equiv V(r)$.
 Since $V$ is a $C^{k-3}$ even function, then Whitney's theorem  ensures that $V(r)=\hat W(r^2)$ for some $C^{[(k-3)/2]}$ function
 $\hat W$, thus proving the desired result.
\end{proof}
\begin{remark}
We note that, obviously, the function $\hat W(t)$ in Corollary~\ref{C:Vey} is of class $C^{k-3}$ for $t>0$.
\end{remark}

\section{Suspension of area-preserving diffeomorphisms}\label{app:suspen}
Let $A(a,b):=\mathbb S^1\times (a,b)$ be an annular domain with $0<a<b<1$, and consider the toroidal manifold $M:=A(0,1)\times \mathbb S^1$. We can endow it with coordinates $(\theta_1,\rho,\theta_2)\in \mathbb S^1\times (0,1)\times \mathbb S^1$ and with the canonical volume form $d\theta_1\wedge d\rho\wedge d\theta_2$. Assume that $w$ is a $C^k$ divergence-free vector field on $M$ that is transverse to the section $\{\theta_2=0\}$. Its first return map at this section defines a $C^k$ diffeomorphism (onto its image) $\cP^w:A(a,b)\to A(0,1)$ that preserves an area form $\mu_2$. The following suspension result is due to D. Treschev~\cite{Treschev}:

\begin{theorem}\label{teo:susp}
Let $\Pi: A(a,b)\to A(0,1)$ be a $C^k$ map that preserves the area $\mu_2$. We assume that $\Pi$ is $C^k$-close to $\cP^w$, i.e. $\|\Pi-\cP^w\|_{C^k(A(a,b))}<\delta$, and that $\Pi=\cP^w$ in a neighborhood of $\partial A(a,b)$. Then if $\delta$ is sufficiently small,
 there exists a $C^k$ divergence-free vector field $\hat w$ on $M$ transverse to the section $\{\theta_2=0\}$ whose Poincar\'e map is $\cP^{\hat w}=\Pi$, $\delta$-close to $w$, that is $\|w-\hat w\|_{C^{k}(M)}<C\delta$, and such that $\hat w=w$ in a neighborhood of $\partial A(a,b)\times\mathbb S^1$.
\end{theorem}


\begin{thebibliography}{1000}



%
%


\bibitem{Adams}
C.C. Adams, \textit{The Knot Book}, AMS, Providence, Rhode Island 2004.

\bibitem{AKh}
V.I. Arnold, B.A. Khesin,
\textit{Topological Methods in Hydrodynamics}, Springer-Verlag,
New York 1998.

\bibitem{BM}
J. Bedrossian, N. Masmoudi, Inviscid damping and the asymptotic stability of planar shear flows in the 2D Euler equations. Publ. Math. Inst. Hautes \'Etudes Sci. 122 (2015) 195--300.

\bibitem{Bo}
J. Bourgain, D. Li, Strong illposedness of the incompressible Euler equation in integer $C^m$ spaces. Geom. Funct. Anal. 25 (2015) 1--86.

\bibitem{C-S}
C.-Q. Cheng, Y.-S. Sun, Existence of invariant tori in three-dimensional measure-preserving maps.
Cel. Mech. Dyn. Ast. 47 (1990) 275--292.

\bibitem{Sverak}
A. Choffrut, V. Sverak, Local structure of the set of steady-state solutions to the 2D incompressible Euler equations.
Geom. Funct. Anal. 22 (2012) 136--201.


\bibitem{Do82}
R. Douady, Une d\'emonstration directe de l'\'equivalence des th\'eor\`emes de tores invariants pour diff\'eomorphismes et champs de vecteurs. C. R. Acad. Sci. Paris 295 (1982) 201--204.


\bibitem{El}
L.H. Eliasson, Normal forms for Hamiltonian systems with Poisson commuting integrals--elliptic case. Comment. Math. Helv. 65 (1990) 4--35.

\bibitem{EM70}
D.G. Ebin, J.E. Marsden, Groups of diffeomorphisms and the motion of an incompressible fluid.
Ann. of Math. 92 (1970) 102--163.

\bibitem{EPT16}
A. Enciso, D. Peralta-Salas, F. Torres de Lizaur, Helicity is the only integral invariant of volume-preserving transformations. Proc. Natl. Acad. Sci. 113 (2016) 2035--2040.
%
%

\bibitem{GH}
J. Guckenheimer, P. Holmes, {\it Nonlinear Oscillations, Dynamical Systems, and Bifurcations of Vector Fields}, Springer-Verlag, New York, 1990.

\bibitem{Hae61}
A. Haefliger, Plongements diff\'erentiables de vari\'et\'es dans vari\'et\'es. Comment. Math. Helv. 36 (1961) 47--82.
%
%
\bibitem{He83}
M. Herman, Sur les courbes invariantes par les diff\'eomorphismes de l'anneau. Ast\'erisque 103-104 (1983) 1--221.


\bibitem{KKP14}
B. Khesin, S. Kuksin, D. Peralta-Salas, KAM theory and the 3D Euler equation. Adv. Math. 267 (2014) 498--522.

\bibitem{KP}
S. Kuksin, G. Perelman, Vey theorem in infinite dimensions and its application to KdV. Discrete Cont. Dyn. Sys. A 27 (2010) 1--24.

\bibitem{Moser}
J. Moser, {\it Stable and random motions in dynamical systems}, Princeton Univ. Press, Princeton, 1973.



\bibitem{Na91}
N. Nadirashvili, Wandering solutions of the two-dimensional Euler equation. Funct. Anal. Appl. 25 (1991) 220--221.

\bibitem{Ro70}
C. Robinson, Generic properties of conservative systems. Amer. J. Math. 92 (1970) 562--603.
%
%
%
%
%


\bibitem{Sh97}
A. Shnirelman, Evolution of singularities, generalized Liapunov function and generalized integral for an ideal incompressible fluid. Amer. J. Math. 119 (1997) 579--608.

\bibitem{Tak}
F. Takens,  A note on sufficiency of jets. Invent. Math. 13 (1971) 225--231.

\bibitem{Treschev}
D. Treschev, Volume preserving diffeomorphisms as Poincare maps for volume preserving flows. Russ. Math. Surv. (2020) to appear.

\bibitem{Vey}
J. Vey, Sur le Lemme de Morse. Invent. Math. 40 (1977) 1--9.

\bibitem{Whi}
H. Whitney, Differentiable even functions. Duke Math. J. 10 (1943) 159--160.
%
%
%
%


\end{thebibliography}
\end{document}